\documentclass[12pt]{amsproc}
\usepackage{amsmath}
\usepackage{amssymb}
\usepackage{amsthm}
\usepackage{enumerate}
\usepackage{booktabs}
\usepackage[colorlinks=true,citecolor=black,linkcolor=black,urlcolor=blue]{hyperref}
\usepackage{anysize}
\marginsize{3cm}{3cm}{2cm}{3cm}

\newcommand{\gauss}[2]{\genfrac{[}{]}{0pt}{}{#1}{#2}}
\newcommand{\gaussm}[1]{[#1]}

\newcommand{\lowerbndqinlem}{3}
\newcommand{\bbF}{\mathbb{F}}
\newcommand{\bbR}{\mathbb{R}}
\newcommand{\scrF}{\mathcal{F}}
\newcommand{\scrP}{\mathcal{P}}

\newcommand{\scrH}{\mathcal{H}}

\newtheorem{theorem}{Theorem}
\newtheorem{lemma}[theorem]{Lemma}

\newtheorem{corollary}[theorem]{Corollary}
\newtheorem{conjecture}[theorem]{Conjecture}

\newtheorem{definition}[theorem]{Definition}

\numberwithin{theorem}{section}
\numberwithin{equation}{section}

\begin{document}


%opening

\title{A Note on the Manickam-Mikl\'{o}s-Singhi Conjecture for Vector Spaces}

\subjclass[2000]{05B25, 51E20, 52C10}

\author{Ferdinand Ihringer}

\address{ %
Mathematisches Institut,
Justus Liebig University Giessen,
Arndtstra\ss{}e 2,
35392 Giessen,
Germany.}
\email{Ferdinand.Ihringer@math.uni-giessen.de}

\begin{abstract}
  Let $V$ be an $n$-dimensional vector space over a finite field with $q$ elements. 
  Define a real-valued weight 
  function on the $1$-dimensional subspaces of $V$ such that the sum of all 
  weights is zero. Let the weight of a subspace $S$ be the sum of the weights
  of the $1$-dimensional subspaces contained in $S$. In 1988 Manickam and Singhi 
  conjectured that if $n \geq 4k$, then the number of $k$-dimensional subspaces with nonnegative weight
  is at least the number of $k$-dimensional subspaces on a fixed $1$-dimensional
  subspace.
  
  Recently, Chowdhury, Huang, Sarkis, Shahriari, and Sudakov proved the conjecture of Manickam 
  and Singhi for $n \geq 3k$. We modify the technique used by Chowdhury, Sarkis, and Shahriari 
  to prove the conjecture for $n \geq 2k$ if $q$ is large. Furthermore, if
  equality holds and $n \geq 2k+1$, then the set of $k$-dimensional subspaces with nonnegative weight 
  is the set of all $k$-dimensional subspaces on a fixed $1$-dimensional subspace.
  With the exception of small $q$, this result is the strongest possible, since the conjecture 
  is no longer true for all $n$ and $k$ with $k < n < 2k$.
  
  \bigskip\noindent \textbf{Keywords:} MMS conjecture; association scheme; EKR theorem
\end{abstract}

\maketitle

\section{Introduction}

In 1984 Thomas Bier introduced the \textit{distribution invariant} of an
association scheme to study certain topological questions related to
so-called net graphs \cite{Bier1984}. This concept was generalized by Bier and Delsarte \cite{Bier1988}, 
where they noticed close connections between the distribution invariant and
other difficult combinatorial problems such as the MDS conjecture,
Erd\H{o}s-Ko-Rado theorems, and various covering problems.  In 1987 Bier and Manickam investigated
the so-called \textit{first distribution invariant} of the 
Johnson scheme \cite{Bier1987}.
In 1988 Manickam, Mikl{\'o}s, and Singhi \cite{Manickam1988, Manickam1988a}
investigated the first distribution invariant of the Johnson scheme further. 
The problem can be paraphrased as follows. Let 
$X$ be a finite set with $n$ elements. Let $f: X \rightarrow \bbR$ be a weight function with 
$\sum_{x \in X} f(x) = 0$. What is the minimum number of $k$-element subsets $S$ 
such that $\sum_{x \in S} f(x)$ is nonnegative? They conjectured the following.

\begin{conjecture}[{\cite[Conjecture 1.4]{Manickam1988a}}]\label{conj_mms_sets}
  Let $n \geq 4k$.
  The number of $k$-element subsets $S$ such that $\sum_{x \in S} f(x)$ is 
  nonnegative is at least $\binom{n-1}{k-1}$.
\end{conjecture}

The Erd\H{o}s-Ko-Rado theorem for sets states that for $n \geq 2k$
the maximum number of $k$-element subsets with pairwise non-trivial 
intersections is at most $\binom{n-1}{k-1}$ \cite{ErdHos1961, Wilson1984}, so there seems to be a connection 
between these two problems.
We shall mention some important works on this conjecture
such as a result by Alon, Huang, 
and Sudakov \cite{Alon2012} who obtained the first 
polynomial bound on $n$ with $n \geq \min\{ 33k^2, 2k^3\}$, 
a result by 
Pokrovskiy \cite{Pokrovskiy2013} who seems to have 
obtained
the first linear bound on $k$ with $n > 10^{46}k$, 
and a result by Chowdhury, Sarkis, and Shahriari \cite{Chowdhury2013}
who did prove the conjecture for $n \geq 8k^2$ and also obtained a result on the analog
problem on vector spaces.
% Conjecture \ref{conj_mms_sets} appears
% to be completely proven in \cite{Blinovsky2014} by Blinovsky.

Manickam and Singhi also considered this analog problem for vector spaces \cite{Manickam1988a}.
In this case the problem can be paraphrased as follows. Let 
$V$ be a finite $n$-dimensional vector space. 
Let $\scrP$ be the set of $1$-dimensional subspaces.
Let $f: \scrP \rightarrow \bbR$ be a weight function with 
$\sum_{P \in \scrP} f(P) = 0$. What is the minimum number of $k$-dimensional subspaces $S$ 
such that $\sum_{P \in S} f(P)$ is nonnegative? They conjectured the following.

\begin{conjecture}[{\cite[Conjecture 1.4]{Manickam1988a}}]\label{conj_mms_vs}
  Let $n \geq 4k$.
  The number of $k$-dimen\-sional subspaces $S$ such that $\sum_{P \in \scrP} f(P)$ is 
  nonnegative is at least the number of $k$-dimen\-sional subspaces on a
  fixed $1$-dimensional subspace.
\end{conjecture}

Manickam and Singhi were able to prove their conjecture if $k$ divides $n$
(which includes $n = k, 2k, 3k$) \cite{Manickam1988a}.
Recently, Chowdhury, Huang, Sarkis, Shahriari, and Sudakov showed that Conjecture \ref{conj_mms_vs}
holds for $n \geq 3k$ \cite{Huang2014, Chowdhury2013}. Hence, technically Conjecture \ref{conj_mms_vs} is proven,
but all known counterexamples satisfy $k < n < 2k$, so it seems reasonable to 
conjecture that only $n \geq 2k$ is necessary. We shall extend the technique 
used by Chowdhury, Sarkis, and Shahriari to show the conjecture for $n \geq 2k$ and (very) large $q$.

\begin{theorem}\label{thm_main}
  Let $V$ be an $n$-dimensional vector space over a finite field $\bbF_q$.
  Let $\scrP$ the set of $1$-dimensional subspaces of $V$.
  Let $f: \scrP \rightarrow \bbR$ be a weighting of the $1$-dimensional subspaces 
  such that $\sum_{P \in \scrP} f(P) = 0$.
  Let $x$ be an integer with $2 \leq x \leq k$.
  If one of the following conditions is satisfied, then there are at least $\gauss{n-1}{k-1}$
  $k$-dimensional subspaces with nonnegative weights.
  \begin{enumerate}[(a)]
   \item $(x-1) n \geq (2x-1) k - x + 2$, $n \geq 2k+2$, and $q \geq (x-1)! \cdot 2^{x+2}$,
   \item $(x-1) n \geq (2x-1) k - x + 1$, $n \geq 2k+1$, and $q \geq (x-1)! \cdot 2^{2x+1}$,
   \item $n \geq 3k$, and $q \geq 2$ or $n = k$, and $q \geq 2$.
  \end{enumerate}
  If equality holds and $n \geq 2k+1$, then the set of nonnegative $k$-dimensional subspaces
  is the set of all $k$-dimensional subspaces on a fixed $1$-dimensional subspace.
\end{theorem}

This implies the following for $x = k$.

\begin{corollary}\label{cor_main}
  Let $V$ be an $n$-dimensional vector space over a finite field $\bbF_q$.
  Let $\scrP$ the set of $1$-dimensional subspaces of $V$.
  Let $f: \scrP \rightarrow \bbR$ be a weighting of the $1$-dimensional subspaces 
  such that $\sum_{P \in \scrP} f(P) = 0$.
  Let $k \geq 2$. There exists a $q_0 \geq 2$ such that the following holds.
  If $n \geq 2k$, and $q \geq q_0$, then 
  there are at least $\gauss{n-1}{k-1}$
  $k$-dimensional subspaces with nonnegative weight.
  If equality holds and $n\geq 2k+1$, then the set of nonnegative $k$-dimensional subspaces
  is the set of all $k$-dimensional subspaces on a fixed $1$-dimensional subspace.
\end{corollary}

\begin{table}
\begin{tabular}{llllll}
$x$ & $2$ & $3$ & $4$ & $5$ & $6$\\ \hline
$n \geq$ & $3k$ & $5k/2$ & $7k/3$ & $9k/4$ & $11k/5$\\ \hline
$q \geq$ & $16$ & $64$ & $384$ & $3072$ & $30720$
\end{tabular}
\caption{\footnotesize The lower bounds for $q$ in Theorem \ref{thm_main} for small $x$ if $n \geq 2k+2$.
  Notice that the case $x=2$ is basically \cite{Chowdhury2013} where a more careful counting achieved $q \geq 2$.}
\end{table}

\begin{table}
\begin{tabular}{llllll}
$x$ & $2$ & $3$ & $4$ & $5$ & $6$\\ \hline
$k \leq$ & $2$ & $4$ & $6$ & $8$ & $10$\\ \hline
$q \geq$ & $32$ & $256$ & $3072$ & $49152$ & $983040$
\end{tabular}
\caption{\footnotesize The lower bounds for $q$ in Theorem \ref{thm_main} for small $k$ if $n = 2k+1$. 
  Note that there is an unpublished result by Chowdhury, Sarkis, and Shahriari that shows 
  the result for $(n, k, q) = (5, 2, q)$, $q > 2$.
  Based on their unpublished notes the author proved the conjecture for $(n, k, q) = (5, 2, q)$.}
\end{table}

Conjecture \ref{conj_mms_vs} has the same connection to Erd\H{o}s-Ko-Rado Theorems
as Conjecture \ref{conj_mms_sets}, since for $n \geq 2k$ 
the largest set of $k$-dimensional subspaces
which pairwise intersect non-trivially is at most the number of $k$-dimensional
subspaces on a fixed $1$-dimensional subspace \cite{Hsieh1975,Frankl1986}.

We shall extend the technique used by Chowdhury et al. Some parts of the proof
are identical. We shall refer to their results to avoid unnecessary repetitions
whenever this is the case. The purpose of the main theorem is to show that Conjecture 
\ref{conj_mms_vs} holds for $n \geq 2k$ if $q$ is large. Often we will ignore 
minor improvements on the condition on $q$ if this would decrease the readability
of the formulas. With the used techniques the lower bound on $q$ 
can not be much better than $(x-1)!$, so the condition on $q$ in Theorem \ref{thm_main}
is reasonably good.

It might be advisable to read Chowdhury et al. \cite{Chowdhury2013} before reading this publication, 
since the structure of the proof in \cite{Chowdhury2013}, which shows the result for $n \geq 3k$,
is roughly the same, but less technical. If one reads this paper with fixed $x=2$, then it is basically
identical to \cite{Chowdhury2013}.

\section{Some Properties of the Gaussian Coefficient}

This section introduces the Gaussian coefficient and some of its properties which 
are needed throughout the following sections. For integers $n, k$ the Gaussian
coefficient is defined as
\begin{align*}
  \gauss{n}{k}_q := \begin{cases}
		   \prod_{i=1}^{k} \frac{q^{n-i+1}-1}{q^i-1} & \text{ if } 0 \leq k \leq n\\
                    0 & \text{ otherwise.}
                  \end{cases}
\end{align*}
Throughout this paper $q$ is fixed, so we will write $\gauss{n}{k}$ instead of $\gauss{n}{k}_q$.
An easy calculation shows for $n \geq k \geq 0$, respectively, $n > k > 0$
\begin{align}
  &\gauss{n}{k} = \gauss{n}{n-k},\label{eq_gauss_compl}\\
  &\gauss{n}{k} = \gauss{n-1}{k} q^k + \gauss{n-1}{k-1} = \gauss{n-1}{k} + \gauss{n-1}{k-1} q^{n-k}. \label{eq_gauss_rec}
\end{align}
A standard double counting argument shows that the number $k$-dimensional subspaces of an
$n$-dimensional vector space over a finite field $\bbF_q$ is $\gauss{n}{k}$.
We will write $\gaussm{n}$ for $\gauss{n}{1}$, the number of $1$-dimensional
subspaces of $V$. If $P$ is a $1$-dimensional subspaces, then the factor space
$V/P$ contains $\gauss{n-1}{k-1}$ $(k-1)$-dimensional subspaces. Hence, exactly
$\gauss{n-1}{k-1}$ $k$-dimensional subspaces of $V$ contain a fixed $1$-dimensional
subspace.

\begin{lemma}\label{lem_stupid_upper_bnd_gauss}
  Let $q \geq 3$, and $n \geq k \geq 0$. Then
   \begin{align*}
      \gauss{n}{k} \leq 2 q^{k(n-k)}.
   \end{align*}
\end{lemma}
\begin{proof}
  By definition,
  \begin{align*}
    \gauss{n}{k} &= \prod_{i=1}^k \frac{q^{n-k+i}-1}{q^i-1} \leq \prod_{i=1}^k \frac{q^{n-k+i}}{q^i-1} = q^{k(n-k)} \prod_{i=1}^k \frac{q^i}{q^i-1}.
  \end{align*}
  Hence,
  \begin{align*}
    \log\left( \prod_{i=1}^\infty \frac{q^i}{q^i-1} \right) &= \sum_{i=1}^\infty \log\left( 1 + \frac{1}{q^i-1} \right)\\
    &\leq \sum_{i=1}^\infty \frac{1}{q^i-1} \leq \frac{1}{2} + \frac{9}{8} \sum_{i=2}^\infty \frac{1}{q^i} = \frac{11}{16}.
  \end{align*}
  This shows the assertion, since $\exp(\frac{11}{16}) < 2$.
\end{proof}

\begin{lemma}\label{lem_bnd_nk}
  For $0 \leq k \leq a \leq n-k$ we find
  \begin{align*}
    &q^{a(k-1)} \gauss{n-a-1}{k-1} \geq \left( 1 - \frac{2}{q^{n-k-a+1}} \right) \gauss{n-1}{k-1}.%,\\
%     &\gauss{n-1}{k-1} \leq q^{k(k-1)} \left(1 + \frac{1}{q^{n-2k+1}}\right) \gauss{n-k-1}{k-1}
  \end{align*}
\end{lemma}
\begin{proof}
  We may assume $k \geq 1$, since in the case $k=0$ both sides of the equation are zero.
  As in \cite[Lemma 3.3]{Chowdhury2013} we have
  \begin{align}
     q^{a(k-1)} \gauss{n-a-1}{k-1} \geq \gauss{n-1}{k-1} - \gaussm{a} \gauss{n-2}{k-2}.\label{eq_bnd_nk_1}
  \end{align}
  Furthermore, the following inequality can be easily verified under the hypothesis and $k \geq 1$.
  \begin{align}
    &\frac{\gaussm{a} \gauss{n-2}{k-2}}{\gauss{n-1}{k-1}} %&= \frac{\gaussm{a} \gaussm{k-1}}{\gaussm{n-1}}
    = \frac{(q^a-1)(q^{k-1}-1)}{(q^{n-1}-1)(q-1)} \notag\\
    &= \frac{2}{q^{n-a-k+1}} - \frac{q^{k+a+1} - 2q^{k+a} + q^{k+1} + q^{a+2} - q^2 - 2q^{-n+k+a+2} + 2q^{-n+k+a+1}}{q(q-1)(q^n-q)} \notag\\
    &\leq \frac{2}{q^{n-a-k+1}}.\label{eq_bnd_nk_2}
  \end{align}
  The equations \eqref{eq_bnd_nk_1} and \eqref{eq_bnd_nk_2} yield the assertion.
\end{proof}

\section{A Bound on Pairwise Intersecting Subspaces}

One crucial ingredient of the result by Chowdhury et al. is
\cite[Lemma 3.6]{Chowdhury2013} which roughly says the following.

\begin{lemma}[{\cite[Lemma 3.6]{Chowdhury2013}}]\label{lem_geom_chowdhury}
  Let $n \geq 2k$.
  Let $A$ and $C$ be two $k$-dimensional subspaces of $V$. If $A \cap C$ is a
  $1$-dimensional subspaces, then the number of $k$-dimensional subspaces of
  $V$ which non-trivially intersect both $A$ and $C$ but do not contain $A \cap C$
  is at most
  \begin{align*}
    \frac{1}{q^{n-3k}} \gauss{n-1}{k-1}.
  \end{align*}
\end{lemma}

It is possible to generalize this statement and we shall do so in this section 
with Lemma \ref{lem_bad_config_intersection_bound}. 
The following lemma is the main improvement of this work over \cite{Chowdhury2013}
while all the other results are merely technically necessary reformulations
of the methods given by Chowdhury et al.

\begin{definition}
  Let $Y$ be a set of subspaces of an $n$-dimensional vector space $V$.
  We say that a subspace $M$ intersects $Y$ badly 
  if all $A \in Y$ satisfy $\dim(A \cap M) = 1$ and
  all $A, B \in Y$ with $A \neq B$ satisfy $\dim(A \cap B \cap M) = 0$.
\end{definition}

\begin{definition}
 We shall call $Y$ a bad configuration if it is a set of $k$-dimensional subspaces 
  of an $n$-dimensional vector space $V$ such that all $C \in Y$ intersect $Y \setminus \{ C \}$ badly.
\end{definition}

\begin{lemma}\label{lem_bad_config_intersection_bound}
  Let $1 < x \leq k$. Let $q \geq 3$.
  Let $n \geq 2k + \delta \geq 2k+1$.
  Suppose $(x-1)n \geq (2x-1)k-x+\delta$.
  Let $q \geq \lowerbndqinlem$.
  Let $Y$ be a bad configuration of $k$-dimensional subspaces of a vector space $V$ with $x = |Y|$.
  Then the number of $k$-dimensional subspaces of $V$ which meet $Y$ badly is at most 
  \begin{align*}
   x^2 \cdot 2^{x} \cdot q^{-\delta} \gauss{n-1}{k-1}.
  \end{align*}
\end{lemma}
\begin{proof}
  Let $Y = \{ A_1, \ldots, A_x\}$ be a bad configuration.
  Let $\tilde{B}$ be a $k$-dimen\-sional subspace of $V$ which meets $Y$ badly.
  Define $S$ as $\langle \tilde{B} \cap A_i: 1 \leq i \leq x \rangle$.
  By the definition of $\tilde{B}$, the subspace $S$ meets $Y$ badly. 
  The subspace $S$ has at most dimension $x$, since it is spanned by $x$ vectors,
  and $S$ has at least dimension $2$, since it intersects $Y$ badly (i.e. the subspaces 
  $S \cap A_i$ are pairwise disjoint).
  Let $m$, $2 \leq m \leq x$, be the dimension of $S$.
  We shall provide upper bounds for the number of choices for $S$ for given $m$ in Part 1. 
  Then, in Part 2, we will bound the number of choices to extend a given $S$
  to a $k$-dimensional subspace of $V$.
  
  \paragraph*{Part 1: The number of choices for a badly intersecting $m$-dimensional subspace.} 
  \textbf{Case $x = m \geq 2$.} We have at most $\gaussm{k}^x$ choices
  for the $1$-dimensional intersections of $S$ with $A_1, \ldots, A_x$, since any $A_i$ contains
  $\gaussm{k}$ $1$-dimensional subspaces.
  
  \textbf{Case $x > m \geq 2$.} 
  Let $M := \{ S \cap A_i \}$.
  By assumption, $S$ has dimension $m$, so there exists a set $B \subseteq M$ with $m$
  elements such that $\langle B \rangle = S$. 
  Hence, we can choose $S$ in $m$ steps by choosing $B$.
  Let $B_i$ be the subset of $B$ after $i$ steps.
  Define $M_i$ as $\{ \langle B_i \rangle \cap A_j: \langle B_i \rangle \cap A_j \text{ non-trivial} \}$.
  We fix the ordering in which we add elements to $B$, i.e. we start by choosing 
  a $1$-dimensional subspace of $A_1$ for $B_1$ and we expand $B_i$ to $B_{i+1}$
  by choosing a $1$-dimensional subspace in $A_j$ where $j$ is the smallest
  possible index such that $\langle B_i \rangle$ and $A_j$ meet non-trivially.
  
  Since $m < x$, we have $|M_{i}| \geq |M_{i-1}| + 2$ for one $i$.
  Let $i_0$ be the first $i$ where this occurs.
  Let $A_j, A_{i_0}$ be two of the elements of $Y$ which meet $\langle B_{i_0-1} \rangle$
  trivially, but $\langle B_{i_0} \rangle$ non-trivially. Notice that $i_0 \in \{ 2, \ldots, m\}$
  and $j \in \{ i_0 + 1, \ldots, m \}$.
  In the following we are going to double count the tuples $(i_0, j, B_1, \ldots, B_{m} = B, S)$
  in accordance with the given defintions.
  
  For given $\{ B_1, \ldots, B_{i-1}\}$, $i \neq i_0$, we have at most $\gaussm{k}$ choices for
  $B_i$, since any $A_i$ contains $\gaussm{k}$ $1$-dimensional subspaces. 
  Hence, we have at most $\gaussm{k}^{m-1}$ choices for all $\{ B_i: 1 \leq i \leq m, i \neq i_0\}$.
  
  We have at most $m - 1 \leq x-2$ choices for $i_0$. We have at most $x-i_0-1$ choices for $j$
  for given $i_0$ as by construction all elements of $A_1, \ldots, A_{i_0-1}$ meet $\langle B_{i_0} \rangle$ non-trivially. 
  Therefore, we have at most $\binom{x-1}{2}$ choices for the pair $(i_0, j)$.
  By our choice of $i_0$, $A_{i_0} \cap \langle B_{i_0} \rangle$ is a subspace of $\langle B_{i_0}, A_j \rangle$.
  By $k = \dim(A_j) = \dim(A_{i_0})$, we have
  \begin{align*}
    &\dim(\langle B_{i_0}, A_j \rangle \cap A_{i_0}) \\
    &\leq \dim(\langle B_{i_0}, A_j \rangle) + \dim(A_{i_0}) - \dim(\langle A_j,  A_{i_0} \rangle)\\
    &\leq (\dim(\langle B_{i_0} \rangle) + \dim(A_j) - 1) + \dim(A_j) - (2\dim(A_j) - 1)\\
    &= \dim(\langle B_{i_0} \rangle) \leq m.
  \end{align*}
  Therefore, we have at most $\gaussm{m}$ choices to extend $B_{i_0-1}$ to $B_{i_0}$
  by choosing a $1$-dimensional subspace in $A_{i_0}$.
  For given $B$, $S$ is uniquely determined by $S = \langle B \rangle$.
  Hence, we have $\binom{x-1}{2} \gaussm{k}^{m-1} \gaussm{m}$ choices for $S$
  for given $(i, j_0, B_1, \ldots, B_m)$.
  
  On the other hand, for $S$ given,
  $B$, $B_i$ and therefore $i_0$ and $j$ are uniquely determined by their
  definitions.
  
  \paragraph*{Part 2: The number of choices for a $k$-dimensional subspace on a given $m$-dimensional subspace.}
  For given $S$ have we have $\gauss{n-m}{k-m}$ choices for a $k$-dimensional subspace
  through $S$. So if $m = x$, then we have at most
  \begin{align}
    \gaussm{k}^x \gauss{n-x}{k-x} \leq 2^{x+1} q^{x(k-1) + (n-k)(k-x)} \label{eq_bad_number1}
  \end{align}
  choices for $B$ by Corollary \ref{lem_stupid_upper_bnd_gauss}.
  If $m < x$, then we have at most
  \begin{align}
    &\gaussm{k}^{m-1} \binom{x-1}{2} \gaussm{m} \gauss{n-m}{k-m} \notag\\&\leq \binom{x-1}{2} \cdot 2^{m+1} q^{(m-1)k + (n-k)(k-m)} \label{eq_bad_number2}
  \end{align}
  choices for $B$ by Corollary \ref{lem_stupid_upper_bnd_gauss}.
  
  By \eqref{eq_bad_number2} and $n \geq 2k+1$ we find that the choices for $B$ for which the dimension
  of $S$ is less than $x$ is at most
  \begin{align*}
    &\sum_{m=2}^{x-1} \binom{x-1}{2} \cdot 2^{m+1} q^{(m-1)k + (n-k)(k-m)} \\
    &\leq \binom{x-1}{2} \left( \sum_{m=2}^{x-1} 2^{m+1} \right) \max_{\scriptscriptstyle m=2, \ldots, x-1} q^{(m-1)k + (n-k)(k-m)}\\
    &\leq \binom{x-1}{2} \cdot 2^{x+1} \max_{m=2, \ldots, x-1} q^{(m-1)k + (n-k)(k-m)}\\
    &= \binom{x-1}{2} \cdot 2^{x+1} q^{k + (n-k)(k-2)}\\
    &= (x-1)(x-2) \cdot 2^{x} q^{k + (n-k)(k-2)}
  \end{align*}
  Hence an upper bound for this number and \eqref{eq_bad_number1} is, by $(x-1)n \geq (2x-1)k-x+\delta$, 
  $n \geq 2k+\delta$, and $(x-1)(x-2) + x \leq x^2$ for $x \geq 2$,
  \begin{align}
    x^2 \cdot 2^{x} q^{\max(k + (n-k)(k-2), x(k-1) + (n-k)(k-x))} \notag\\ \leq x^2 \cdot 2^{x} q^{-\delta + (n-k)(k-1)}. \label{eq_bad_number3}
  \end{align}
  Applying the inequality $\gauss{n-1}{k-1} > q^{(k-1)(n-k)}$ to \eqref{eq_bad_number3} shows the assertion.
\end{proof}

\section{An Eigenvalue Technique}

In this section we shall restate the arguments used in Section 3 of \cite{Chowdhury2013}.
We shall include proofs for the results if these results extend results 
of \cite{Chowdhury2013} in some way.
Otherwise we will just refer to \cite{Chowdhury2013}.
Before we do this we want to give some context to the used eigenvalue technique.

Let $V$ be a vector space over $\bbF_q$.
Let $f: \scrP \rightarrow \bbR$ a weighting of $\scrP$ with 
$\sum_{P \in \scrP} f(P) = 0$. 
We suppose $f \not\equiv 0$ throughout this section, since the case $f \equiv 0$ is trivial.
We say that two subspaces $R$ and $S$ are \textit{incident} if $S \subseteq R$
or $R \subseteq S$.
Define the \textit{incidence matrix} $W_{ij}$
as the matrix whose rows are indexed by the $i$-dimensional subspaces of $V$,
whose columns are indexed by the $j$-dimensional subspaces of $V$, by
\begin{align*}
  (W_{ij})_{RS} = 
  \begin{cases}
    1 & \text{ if } S \text{ is incident with } R,\\
    0 & \text{ otherwise.}
  \end{cases}
\end{align*}
Define the \textit{adjacency matrix} $A_i$ as the matrix whose rows and columns are
indexed by the $k$-dimensional subspaces of $V$ by
\begin{align*}
  (A_i)_{RS} =
  \begin{cases}
    1 & \text{ if } \dim(S \cap R) = k - i,\\
    0 & \text{ otherwise.}
  \end{cases}
\end{align*}
We write $b_S$ for the weight of a $k$-dimensional subspace of $V$ 
(i.e. $b_S = \sum_{P \in S} f(P)$).
By the definition of the weight of $S$, clearly $b = W_{k1} f$ holds if
we consider $f = (f(P))$ as a vector indexed by the $1$-dimensional subspaces $P$ of $V$
and $b = (b_S)$ as a vector indexed by the $k$-dimensional subspaces $S$ of $V$.
It was shown by Frankl and Wilson \cite{Frankl1986} that the set of all sets of 
$k$-dimensional subspaces on a fixed $1$-dimensional subspace spans
the orthogonal sum of two eigenspaces of $A_i$ of the form $\langle j \rangle \perp \tilde{V}$, 
where $j$ is the all-one vector. This implies that $b = W_{k1} f$ is an 
eigenvector of $A_i$ and lies in the eigenspace $\tilde{V}$. The corresponding
eigenvalues can be found the literature \cite{Delsarte1976,Eisfeld1999}.
We shall use the PhD thesis
of Fr\'{e}d\'{e}ric Vanhove \cite{Vanhove2011} as reference.
In view of \cite[Theorem 3.2.4, Remark 3.2.5]{Vanhove2011}, $b$ is an 
eigenvector of the eigenspace $V_1^k$ (i.e. $\tilde{V} = V_1^k$ in the notation of \cite{Vanhove2011}) which leads to the following result.

\begin{lemma}
  Let $A_i$ be the distance-$i$ adjacency matrix of the $k$-dimensional subspaces
  of $V$. Let $b$ be the weight vector of the $k$-dimensional subspaces of $V$.
  Then $b$ is an eigenvector of $A_i$ with eigenvalue
  \begin{align*}
    \gauss{n-k-1}{i} \gauss{k-1}{i}  q^{(i+1)i}
      - \gauss{n-k-1}{i-1} \gauss{k-1}{i-1} q^{i(i-1)}.
  \end{align*}
\end{lemma}
\begin{proof}
  By \cite[Remark 3.2.5]{Vanhove2011}, \eqref{eq_gauss_compl}, and \eqref{eq_gauss_rec},
  \begin{align*}
    A_ib &= \gauss{n-k}{n-k-i} \gauss{k-1}{i} q^{i^2}
      - \gauss{n-k-1}{n-k-i} \gauss{k}{i} q^{i(i-1)}\\
      &= \gauss{n-k}{i} \gauss{k-1}{i} q^{i^2}
      - \gauss{n-k-1}{i-1} \gauss{k}{i} q^{i(i-1)}\\
      &= \gauss{n-k-1}{i} \gauss{k-1}{i} q^{(i+1)i}
      + \gauss{n-k-1}{i-1} \gauss{k-1}{i} q^{i^2}
      - \gauss{n-k-1}{i-1} \gauss{k}{i} q^{i(i-1)}\\
      &= \gauss{n-k-1}{i} \gauss{k-1}{i}  q^{(i+1)i}
      - \gauss{n-k-1}{i-1} \gauss{k-1}{i-1} q^{i(i-1)}.
  \end{align*}
\end{proof}

For a given $k$-dimensional subspaces $C$ we have
\begin{align*}
  &\sum_{\dim(S \cap C) = k-i} b_S = (A_i b)_{C} \\
  &= 
  \left( \gauss{n-k-1}{i} \gauss{k-1}{i}  q^{(i+1)i}
      - \gauss{n-k-1}{i-1} \gauss{k-1}{i-1} q^{i(i-1)} \right) b_C,
\end{align*}
which makes these eigenvalues very useful.
In particular, for $A_{k-1}$ we get the following result.
Notice that this number was directly calculated by Chowdhury et al. in \cite[Equation (3.15)]{Chowdhury2013}
and that we adopted their presentation of the formula.

\begin{lemma}\label{lem_ev_A1}
  Let $n \geq 2k$.
  Let $C$ be a $k$-dimensional subspace of $V$. Then we have
  \begin{align*}
    \sum_{\dim(S \cap C) = 1} b_S = \left( q^{k(k-1)} \gauss{n-k-1}{k-1} 
    - q^{(k-1)(k-2)} \gaussm{k-1} \gauss{n-k-1}{k-2}\right) b_C.
  \end{align*}
\end{lemma}

Let $A$ be one of the $k$-dimensional subspaces with $b_A = \max b_S$ (i.e. a 
$k$-dimensional subspace with the highest weight).
The idea of this section is
to reach a situation where we can apply Lemma \ref{lem_bad_config_intersection_bound}
on a large bad configuration.
We shall do this in several steps. In Lemma \ref{lem_lower_bounds_bCi} we show
that we are able to find a lot of nonnegative $k$-dimensional subspaces which intersect $A$
in a $1$-dimensional subspace and have a weight of nearly $b_A$ for large $q$.
Lemma \ref{lem_all_M_intersection_bnd} then shows that many of these nonnegative
$k$-dimensional subspaces pairwise intersect in exactly a $1$-dimensional subspace,
which leads to a situation where we can apply Lemma \ref{lem_bad_config_intersection_bound}.

\begin{lemma}\label{lem_A1_intersection}
  Let $n \geq 2k+1$.
  Let $A$ denote a highest weight $k$-dimensional subspace of $V$.
  Let $C$ be a nonnegative $k$-dimensional subspace of $V$. Then at least
  \begin{align*}
    \left( 1 - \frac{3}{q^{n-2k+1}} \right) \gauss{n-1}{k-1} \frac{b_C}{b_A}
  \end{align*}
  nonnegative $k$-dimensional subspaces intersect $C$ in exactly a $1$-dimensional subspace.
\end{lemma}
\begin{proof}
  By Lemma \ref{lem_ev_A1}, $\sum_{\dim(S \cap C) = 1} b_S = $
  \begin{align*}
    \left( q^{k(k-1)} \gauss{n-k-1}{k-1} 
    - q^{(k-1)(k-2)} \gaussm{k-1} \gauss{n-k-1}{k-2}\right) b_C.
  \end{align*}
  Each $b_S$ is less than or equal to $b_A$ which yields at least
  \begin{align*}
    \left( q^{k(k-1)} \gauss{n-k-1}{k-1} 
    - q^{(k-1)(k-2)} \gaussm{k-1} \gauss{n-k-1}{k-2}\right) \frac{b_C}{b_A}
  \end{align*}
  nonnegative $k$-dimensional subspaces that intersect $C$ in exactly a $1$-dimensional subspace.
  As $n \geq 2k+1$, we have
  \begin{align}
    &\frac{q^{(k-1)(k-2)} \gaussm{k-1} \gauss{n-k-1}{k-2}}{\gauss{n-1}{k-1}}\label{eq_ac_km_bnd}\\ \nonumber
    &= q^{(k-1)(k-2)} \gaussm{k-1} \frac{q^{k-1}-1}{q^{n-k+1}-1} \prod_{i=1}^{k-2} \frac{q^{n-k-i}-1}{q^{n-i}-1}\\ \nonumber
    &< q^{(k-1)^2} q^{-n+2k-2} q^{-(k-2)k}\\ \nonumber
    &\leq \frac{1}{q^{n-2k+1}}.
  \end{align}
  Then Lemma \ref{lem_bnd_nk} (with $a=k$) shows the assertion.
\end{proof}

\begin{lemma}\label{lem_lower_bounds_bCi}
  Let $n \geq 2k+1$. Let $c$ be a real number with $3 \leq c \leq q$.
  Let $A$ denote a highest weight $k$-dimensional subspace of $V$.
  Let $C_i$ denote the $i$-th highest weight $k$-dimensional subspace of $V$ such 
  that $\dim(A \cap C_i) = 1$. Suppose $i \leq \frac{c-3}{q} \gauss{n-1}{k-1} + 1$, and
  suppose that there are at most $\gauss{n-1}{k-1}$ nonnegative
  $k$-dimensional subspaces of $V$, then $b_{C_i}$, the weight of $C_i$, satisfies 
  \begin{align*}
    b_{C_i} > \left( 1 - \frac{c}{q} \right) b_A
  \end{align*}
\end{lemma}
\begin{proof}
  By Lemma \ref{lem_ev_A1} and $b_{C_i} \leq b_A$, we have
  \begin{align*}
    &\sum_{j \geq i} b_{C_j} = \sum b_{C_j} - \sum_{j < i} b_{C_j}
    \geq b_A \cdot \\ & \left( q^{k(k-1)} \gauss{n-k-1}{k-1} 
    - q^{(k-1)(k-2)} \gaussm{k-1} \gauss{n-k-1}{k-2} - i + 1\right).
  \end{align*}
  We suppose that we have at most $\gauss{n-1}{k-1}$ nonnegative $k$-dimensional
  subspaces. Hence, we find
  \begin{align}
    b_{C_i} \geq \frac{q^{k(k-1)} \gauss{n-k-1}{k-1} 
    - q^{(k-1)(k-2)} \gaussm{k-1} \gauss{n-k-1}{k-2} - i + 1}{\gauss{n-1}{k-1}}.\label{eq_bCi_bnd}
  \end{align}
  By hypothesis $i \leq \frac{c-3}{q} \gauss{n-1}{k-1} + 1 \leq \left(\frac{c}{q} - \frac{3}{q^{n-2k+1}}\right) \gauss{n-1}{k-1} + 1$, 
  $n \geq 2k+1$, and \eqref{eq_ac_km_bnd},
  we have
  \begin{align}
    \nonumber &q^{(k-1)(k-2)} \gaussm{k-1} \gauss{n-k-1}{k-2} + i - 1 \\&\leq \left(\frac{c}{q} - \frac{2}{q^{n-2k+1}}\right) \gauss{n-1}{k-1}.\label{eq_bCi_final}
  \end{align}
  By Lemma \ref{lem_bnd_nk}, 
  \begin{align*}
     q^{k(k-1)} \gauss{n-k-1}{k-1} \geq \left( 1 - \frac{2}{q^{n-2k+1}} \right) \gauss{n-1}{k-1}.
  \end{align*}
  Hence, \eqref{eq_bCi_final} and \eqref{eq_bCi_bnd} yield assertion.
\end{proof}

\begin{lemma}\label{lem_all_M_intersection_bnd}
  Let $n \geq 2k+\delta \geq 2k+1$. Let $c$ be a real number with $3 \leq c \leq q$.
  Let $q \geq 3$.
  Let $A$ denote a highest weight $k$-dimensional subspace of $V$.
  Let $C_i$ denote the $i$-th highest weight $k$-dimensional subspace of $V$ such 
  that $\dim(A \cap C_i) = 1$.
  Let $I$ be a subset of $\{ 1, \ldots, \lfloor \frac{c-3}{q} \gauss{n-1}{k-1} + 1 \rfloor\}$ with $|I| \leq k-1$.
  Set $x = |I|+1$.
  Set $M := \{A \} \cup \{ C_i: i \in I \}$.
  Suppose that there are at most $\gauss{n-1}{k-1}$ nonnegative $k$-dimensional 
  subspaces. 
  Then we have the following.
  \begin{enumerate}[(a)]
   \item At least 
    \begin{align*}
      \left( 1 - \frac{(x-1)c}{q} - \frac{3x}{q^{n-2k+1}} \right) \gauss{n-1}{k-1}
    \end{align*}
  nonnegative $k$-dimensional subspaces intersect each element of $M$ in exactly a
  $1$-dimensional subspace.
   \item Suppose that $M$ is a bad configuration. Suppose $x > 1$ with $(x-1)n \geq (2x-1)k-x+\delta$.
  Then there exist $S, R \in M$ such that the $1$-dimensional subspace $S \cap R$
  lies in at least
  \begin{align*}
    \left( 1 - \frac{(x-1)c}{q} - \frac{3x}{q^{n-2k+1}} 
    - x^2 \cdot 2^{x} q^{-\delta} \right) \frac{\gauss{n-1}{k-1}}{\binom{x}{2}}
  \end{align*}
 nonnegative $k$-dimensional subspaces.
  \end{enumerate}
\end{lemma}
\begin{proof}
  We write $M = \{ M_1, \ldots, M_x \}$ with $M_1 = A$.
  
  First we shall show by induction on $x \geq 1$ that at least
  \begin{align}
    \left( 1 - \frac{(x-1)c}{q} - \frac{3x}{q^{n-2k+1}} \right) \gauss{n-1}{k-1}\label{eq_induction_hyp}
  \end{align}
  nonnegative $k$-subspaces intersect all elements of $\{ M_1, \ldots, M_{x} \}$
  in exactly a $1$-dimensional subspace.
  
  For $A$ we find by applying Lemma \ref{lem_A1_intersection} that $A$ meets at least
  \begin{align}
    \left( 1 - \frac{3}{q^{n-2k+1}} \right) \gauss{n-1}{k-1}
  \end{align}
  nonnegative $k$-dimensional subspaces in a $1$-dimensional subspace. This shows \eqref{eq_induction_hyp} for $x=1$.
  
  Now suppose $x > 1$. 
  By hypothesis, there are at most $\gauss{n-1}{k-1}$ nonnegative $k$-dimensional subspaces.
  So by Lemma \ref{lem_A1_intersection}, Lemma \ref{lem_lower_bounds_bCi}, \eqref{eq_induction_hyp}, 
  and the sieve principle, we find that at least
    \begin{align}
    &\left( 1 - \frac{(x-1)c}{q} - \frac{3x}{q^{n-2k+1}} \right) \gauss{n-1}{k-1}\notag\\ 
    + &\left( 1 - \frac{c}{q} - \frac{3}{q^{n-2k+1}} \right) \gauss{n-1}{k-1} - \gauss{n-1}{k-1}\label{eq_lower_bnd_Ci_A1}\\
    \geq &\left( 1 - \frac{xc}{q} - \frac{3(x+1)}{q^{n-2k+1}} \right) \gauss{n-1}{k-1}\notag
    \end{align}
  nonnegative $k$-dimensional subspaces intersect all elements of the set $\{ M_1, \ldots, M_x \}$
  in exactly a $1$-dimensional subspace. This shows (a).
  
  By Lemma \ref{lem_bad_config_intersection_bound} and \eqref{eq_induction_hyp},
  we have that at least
  \begin{align}
    \left( 1 - \frac{(x-1)c}{q} - \frac{3x}{q^{n-2k+1}} 
    - x^2 \cdot 2^{x} q^{-\delta} \right) \gauss{n-1}{k-1}
  \end{align}
  nonnegative $k$-dimensional subspaces intersect all elements of $M$ in exactly a
  $1$-dimensional subspace and contain a $1$-dimensional subspace of the form 
  $M_i \cap M_j$, $i \neq j$. As we have at most $\binom{x}{2}$ such $1$-dimensional 
  subspaces $M_i \cap M_j$, this shows (b).
\end{proof}

\begin{lemma}\label{lem_there_exists_bad_configuration}
  Let $n \geq 2k+1$. Let $x$ be a number with $2 \leq x \leq k$. Let $q > (x-1)! \cdot 2^{x+2}$.
  Let $A$ denote a highest weight $k$-dimensional subspace of $V$.
  Let $C_i$ denote the $i$-th highest weight $k$-dimensional subspace of $V$ such 
  that $\dim(A \cap C_i) = 1$.
  Suppose that there are at most $\gauss{n-1}{k-1}$ nonnegative $k$-dimensional 
  subspaces of $V$.
  Suppose that no $1$-dimensional subspace is contained in more than 
  $\frac{3}{q} \gauss{n-1}{k-1}$ nonnegative $k$-dimensional subspaces.
  Then $\tilde{M}_{x} := \{ A \} \cup \{ C_i: i \leq \frac{(x-2)! \cdot 2^{x+1}-3}{q} \gauss{n-1}{k-1} + 1\}$ contains
  a bad configuration $M$ with $x$ elements and $A \in M$.
\end{lemma}
\begin{proof}
  We shall prove our claim by induction on $x$.
  If $x = 1$, then $\{ A \}$ is a bad $x$-configuration.
  If $x = 2$, then $\{ A, C_1 \}$ is a bad $x$-configuration.
  
  Only the case $x > 2$ remains.
  Suppose that $\tilde{M}_{x}$
  contains a bad $x$-configuration $M = \{ M_1, \ldots, M_x \}$ with $A \in M$ and $x \geq 2$. 
  By Lemma \ref{lem_all_M_intersection_bnd} (a), at least 
  \begin{align*}
      \alpha := \left( 1 - \frac{(x-1)! \cdot 2^{x+1} }{q} - \frac{3x}{q^{n-2k+1}} \right) \gauss{n-1}{k-1}
  \end{align*}
  nonnegative $k$-dimensional subspaces intersect all elements of $M$ in exactly a
  $1$-dimensional subspace. By hypothesis, at most
  \begin{align*}
    \frac{3}{q} \binom{x}{2} \gauss{n-1}{k-1}
  \end{align*}
  nonnegative $k$-dimensional subspaces of $V$ contain one of the 
  $\binom{x}{2}$ $1$-dimensional subspaces
  $M_i \cap M_j$, $i \neq j$. So the number of nonnegative 
  $k$-dimensional subspaces which meet $M$ badly is, by definition, by $q \geq (x-1)! \cdot 2^{x+2}$, and by $n-2k+1 \geq 2$, at least
  \begin{align*}
    &\alpha - \frac{3}{q} \binom{x}{2} \gauss{n-1}{k-1}\\
    &= \left( 1 - \frac{(x-1)!  \cdot 2^{x+1} + 3\binom{x}{2} }{q} - \frac{3x}{q^{n-2k+1}} \right) \gauss{n-1}{k-1}\\
    &\geq \left( 1 - \frac{(x-1)! \cdot 2^{x+1} + 3\binom{x}{2} + 3x/q}{q} \right) \gauss{n-1}{k-1}\\
    &\geq \left( 1 - \frac{(x-1)!  \cdot 2^{x+1} + 3\binom{x}{2}+1}{q}\right) \gauss{n-1}{k-1} =: \beta.
  \end{align*}
  There are at most $\gauss{n-1}{k-1}$ nonnegative $k$-dimensional subspaces and,
  by Lemma \ref{lem_lower_bounds_bCi} and $q > (x-1)! \cdot 2^{x+2}$, all elements of
  $\tilde{M}_{x+1}$ have nonnegative weight, so
  at least
  \begin{align*}
    &\beta + |\tilde{M}_{x+1}| - \gauss{n-1}{k-1}\\
    \geq &\left( \frac{(x-1)! \cdot 2^{x+2} - 3 - (x-1)! \cdot 2^{x+2} - 3\binom{x}{2} - 1}{q} \right) \gauss{n-1}{k-1}\\
      = &\left( \frac{(x-1)! \cdot 2^{x+1} - 4 - 3\binom{x}{2}}{q} \right) \gauss{n-1}{k-1}
  \end{align*}
  nonnegative $k$-dimensional subspaces meet $M$ badly and are in $\tilde{M}_{x+1}$.
  For $x \geq 2$ this number is positive, so we will find a bad configuration of
  nonnegative $k$-dimensional subspaces in $\tilde{M}_{x+1}$.
\end{proof}

\section{Chowdhury, Sarkis, and Shahriari's Averaging Bound}

In \cite[Lemma 4.5]{Chowdhury2013} Chowdhury, Sarkis, and Shahriari apply a result 
by Beutelspacher \cite{Beutelspacher1979} on partial spreads of projective spaces.
They do not fully state what their proof shows which is why we have to restate 
their result here in a bit more detail. See \cite{Chowdhury2013} for the complete argument.

\begin{lemma}\label{lem_imp_section4_lemma}
  If $n = 2k+\delta$ with $0 \leq \delta < k$, and $T$ is a negative weight $k$-dimensional subspace, then there are at least 
  \begin{align*}
    \left( 1- \frac{2}{q} \right) \gauss{n-1}{k-1}
  \end{align*}
  nonnegative $k$-dimensional subspaces that have trivial intersection with $T$.
\end{lemma}
\begin{proof}
  The proof is as in \cite[Lemma 4.5]{Chowdhury2013} with the exception of 
  \cite[Equation (4.51)]{Chowdhury2013}. By Lemma \ref{lem_bnd_nk}, we have for $n = 2k+\delta$
  \begin{align*}
   |\scrF| \geq q^{(k+\delta)(k-1)} \gauss{n-k-\delta-1}{k-1} &= q^{(k+\delta)(k-1)} \geq \left( 1 - \frac{2}{q} \right) \gauss{n-1}{k-1}.
  \end{align*}
  which yields the assertion.
\end{proof}

\section{Proof of Theorem \ref{thm_main}}

\begin{proof}[Proof of Theorem \ref{thm_main}]
  We may assume that there are at most $\gauss{n-1}{k-1}$ nonnegative $k$-dimensional
  subspaces. We will also assume $2k < n < 3k$, 
  since the remaining cases are covered in \cite[Theorem 1.3]{Chowdhury2013} and \cite{Manickam1988a}.
  Also notice that the theorem reuqires $x \geq 2$.
  
  If there exists a $1$-dimensional subspace $P$ which is contained
  in $\gauss{n-1}{k-1}$ nonnegative $k$-dimensional subspaces, then we are done.
  Therefore, we can suppose that all $1$-dimensional subspaces are contained in at least one
  $k$-dimensional subspace with negative weight. 
  
  Suppose there exists a $1$-dimensional subspace $P$ which is contained in more than
  $\frac{2}{q} \gauss{n-1}{k-1}$ nonnegative $k$-dimensional subspaces. There 
  exists a negative $k$-dimensional subspace $T$ on $P$, so there are at least
  \begin{align*}
    \left( 1- \frac{2}{q} \right) \gauss{n-1}{k-1}
  \end{align*}
  nonnegative $k$-dimensional subspaces not on $P$ by Lemma \ref{lem_imp_section4_lemma}.
  Then there are more than $\gauss{n-1}{k-1}+1$ nonnegative $k$-dimensional subspaces
  which contradicts our assumption.
  
  Therefore no $1$-dimensional subspace is contained in more than 
  $\frac{2}{q} \gauss{n-1}{k-1}$ nonnegative $k$-dimensional subspaces.
  Let $A$ denote a heighest weight $k$-dimensional subspace of $V$.
  Let $C_i$ denote the $i$-th highest weight $k$-dimensional subspace of $V$ such 
  that $\dim(A \cap C_i) = 1$. 
  By Lemma \ref{lem_there_exists_bad_configuration}, there exists a bad configuration
  in $\{ A \} \cup \{ C_i: i \leq \frac{(x-2)! \cdot 2^{x+2}-3}{q} \gauss{n-1}{k-1} + 1\}$
  with $x$ elements.
  Hence, we can apply Lemma \ref{lem_all_M_intersection_bnd} (b) with 
  \begin{align*}
    c = (x-2)! \cdot 2^{x+1}
  \end{align*}
  which shows that we find a $1$-dimensional subspace that is a subspace of at least
  \begin{align}
    \left( 1 - \frac{(x-1)! \cdot 2^{x+1}}{q} - \frac{3x}{q^{n-2k+1}} 
    - x^2 \cdot 2^{x} \cdot q^{-\delta} \right)
    \frac{\gauss{n-1}{k-1}}{\binom{x}{2}}\label{eq_viele_viele_subs}
  \end{align}
  nonnegative $k$-dimensional subspaces where $\delta = 2$ in Case (a), and 
  $\delta = 1$ in Case (b).
  
  If $\delta = 2$, then the assumptions $q \geq (x-1)! \cdot 2^{x+2}$,
  $(x-1) n \geq (2x-1) k - x + 2$ (particularly, $x > 1$), and $n \geq 2k+2$ imply
  \begin{align*}
    &\frac{(x-1)! \cdot 2^{x+1}}{q} \leq \frac{1}{2},\\
    &\frac{3x}{q^{n-2k+1}} \leq \frac{3x}{q^3} \leq \frac{3x}{2^{3(x+2)}} \leq \frac{3}{256},\\
    &x^2 \cdot 2^{x} \cdot q^{-\delta} \leq \frac{x^2 \cdot 2^x}{\left(2^{x+2}\right)^2} \leq \frac{x^2}{2^{x+4}} \leq \frac{1}{8},\\
    &\frac{q}{\binom{x}{2}} \geq \frac{ (x-1)! \cdot 2^{x+3}}{x(x-1)} \geq 16,
  \end{align*}
  so \eqref{eq_viele_viele_subs} is at least
  \begin{align*}
    \left( 1 - \frac{1}{2} - \frac{3}{256} - \frac{1}{8} \right) \frac{\gauss{n-1}{k-1}}{\binom{x}{2}}
    &= \frac{93}{256 \binom{x}{2}} \gauss{n-1}{k-1}\\
    &\geq \frac{93}{16q} \gauss{n-1}{k-1} > \frac{2}{q} \gauss{n-1}{k-1}.
  \end{align*}
  This contradicts our assumption that no $1$-dimensional subspace $P$ is contained
  in more than $\frac{2}{q} \gauss{n-1}{k-1}$ nonnegative $k$-dimensional subspaces.
  Hence, Part (a) of the theorem follows.
  
  If $\delta = 1$, then the assumptions $q \geq (x-1)! \cdot 2^{2x+1}$,
  $(x-1) n \geq (2x-1) k - x + 1$, and $n \geq 2k+1$ imply Part (b) of
  the theorem with similar calculations. Here we have
  \begin{align*}
    &\frac{(x-1)! \cdot 2^{x+1}}{q} \leq \frac{1}{4},\\
    &\frac{3x}{q^{n-2k+1}} \leq \frac{3x}{q^3} \leq \frac{3x}{2^{3(2x+1)}} \leq \frac{3}{1024},\\
    &x^2 \cdot 2^{x} \cdot q^{-\delta} \leq \frac{x^2 \cdot 2^x}{2^{2x+1}} \leq \frac{x^2}{2^{x+1}} \leq \frac{9}{16},\\
    &\frac{q}{\binom{x}{2}} \geq \frac{ (x-1)! \cdot 2^{2x+2}}{x(x-1)} \geq 32.
  \end{align*}
  Then \eqref{eq_viele_viele_subs} is at least
  \begin{align*}
    \left( 1 - \frac{1}{4} - \frac{3}{1024} - \frac{9}{16} \right) \frac{\gauss{n-1}{k-1}}{\binom{x}{2}}
    &= \frac{189}{1024 \binom{x}{2}} \gauss{n-1}{k-1}\\
    &\geq \frac{189}{32} \gauss{n-1}{k-1} > \frac{2}{q} \gauss{n-1}{k-1}.
  \end{align*}
\end{proof}

\section{Duality}

For the sake of completeness we also mention the following simple exercise.

\begin{lemma}\label{prop_dual}
  Let $n \geq 2k$.
  If there are at least $\alpha$
  $(n-k)$-dimensional subspaces with nonnegative weight, then there are at least
  $\alpha$ $k$-dimensional subspaces with nonnegative weight.
  Furthermore, the set of $k$-dimensional subspaces with nonnegative weights
  is isomorphic to a dual of the set of nonnegative $(n-k)$-dimensional subspaces.
\end{lemma}
\begin{proof}
  Let $\scrH$ be the set of hyperplanes of $V$. Define the weight function 
  $g: \scrH \rightarrow \bbR$ by $g(H) = \sum_{P \in H} f(P)$. Define the
  $g$-weight of a $k$-dimensional subspace $U$ by $g(U) = \sum_{U \subseteq H} g(H)$.
  Furthermore by $\sum_{P \in \scrP} f(P) = 0$,
  \begin{align*}
    g(U) &= \sum_{U \subseteq H} g(H)\\
    &= \sum_{U \subseteq H} \sum_{P \in H} f(P)\\
    &= \sum_{U \subseteq H} \left(\sum_{P \in H \cap U} f(P)\right) + \left(\sum_{P \in H \setminus U} f(P)\right)\\
    &= \sum_{P \in \scrP \cap U} (\gaussm{n-k} f(P) + \sum_{P \in \scrP \setminus U} \gaussm{n-k-1} f(P))\\
    &= q^{n-k-1} f(U).
  \end{align*}
  Hence, we can consider the problem in the dual vector space of $V$ (which is isomorphic to $V$)
  with $g$ as the weight function on points (of the dual space).
  Then the assertion is obvious.
\end{proof}

This shows that one only has to investigate the MMS problem for $n \geq 2k$: If $n < 2k$,
then $n > n + (n-2k) = 2(n-k)$.

\section{Some Concluding Remarks}

The used argument is based on the observation that the only set of nonnegative $k$-dimensional
subspaces which reaches the bound $\gauss{n-1}{k-1}$ seems to be the set of all generators
on a fixed $1$-dimensional subspace. This can no longer work for $n=2k$, since
one can construct another example of that size as follows.
Fix a $(2k-1)$-dimensional subspace $S$, put the weight $-1$ on all
$1$-dimensional subspaces not in $S$ and the weight $q^{2k-1}/\gaussm{2k-1}$ 
on all $1$-dimensional subspaces in $S$.
Then exactly the $\gauss{n-1}{k-1}$ $k$-dimensional subspaces in $S$ are the 
nonnegative ones, so this is a second example. This is in fact the only other example
in this case (see below).

Conjecture \ref{conj_mms_vs} is wrong for $k < n < 2k$ as one can see by a similar 
example which we obtain by duality: Fix a $(n-1)$-dimensional subspace $S$, put the weight $-1$ on all 
$1$-dimensional subspaces not in $S$, put the weight $q^{n-1}/\gaussm{n-1}$ on
all $1$-dimensional subspaces in $S$. Then the nonnegative $k$-dimensional subspaces
are exactly the $k$-dimensional subspaces in $S$. There are $\gauss{n-1}{k}$ such
subspaces, so $\gauss{n-1}{k} < \gauss{n-1}{k-1}$ for $k < n < 2k$ shows that
Conjecture \ref{conj_mms_vs} does not hold in this range.

As the cases $n < 2k$ are covered by Lemma \ref{prop_dual}, it seems to be reasonable
to conjecture the following.
\begin{enumerate}[(a)]
 \item for $k < n < 2k$, the minimum number of nonnegative $k$-dimensional
subspaces is $\gauss{n-1}{k}$ with equality for the example given above,
\item for $n = 2k$, the minimum number of nonnegative $k$-dimensional subspaces 
 is $\gauss{n-1}{k-1}$ with equality for the two given example, i.e. either all
 nonnegative $k$-dimensional subspaces contain a fixed $1$-dimensional subspace
 or all nonnegative $k$-dimensional subspaces are contained in a fixed 
 $(n-1)$-dimensional subspace,
 \item for $n > 2k$, the minimum number of nonnegative $k$-dimensional subspaces 
 is $\gauss{n-1}{k-1}$ with equality if and only if all
 nonnegative $k$-dimensional subspaces contain a fixed $1$-dimensional subspace.
\end{enumerate}
Notice that (b) is implied by the proof of \cite[Theorem 3.1]{Manickam1988a} and
the classification of all Erd\H{o}s-Ko-Rado sets of size $\gauss{n-1}{k-1}$.
Ameera Chowdhury remarked\footnote{Private communication.}
that this conjectures is the canonical generalization of a
conjecture on the MMS problem for sets given in \cite{Aydinian2012, Blinovsky2014} which was confirmed
for small cases in \cite{Hartke2014}. Additionally, the author did a (non-exhaustive) computer search for weightings with
a minimum number of nonnegative $k$-dimensional subspaces which support the stated conjecture on vector spaces.

\section*{Acknowledgment} The author would like to thank Ameera Chowdhury 
for her many very helpful remarks, Simeon Ball for telling him about the problem, 
and Klaus Metsch for his careful proofreading of the mathematical content.

\bibliographystyle{plain}
% \bibliography{mms_problem}
% \bibliography{mms_problem_vs_complete_v13_pub_arx}

\begin{thebibliography}{10}

\bibitem{Alon2012}
Noga Alon, Hao Huang, and Benny Sudakov.
\newblock Nonnegative {$k$}-sums, fractional covers, and probability of small
  deviations.
\newblock {\em J. Combin. Theory Ser. B}, 102(3):784--796, 2012.

\bibitem{Aydinian2012}
H.~Aydinian and V.~Blinovsky.
\newblock A remark on the problem of nonnegative k-subset sums.
\newblock {\em Problems of Information Transmission}, 48(4):347--351, 2012.

\bibitem{Beutelspacher1979}
Albrecht Beutelspacher.
\newblock On {$t$}-covers in finite projective spaces.
\newblock {\em J. Geom.}, 12(1):10--16, 1979.

\bibitem{Bier1984}
Thomas Bier.
\newblock A distribution invariant for association schemes and strongly regular
  graphs.
\newblock {\em Linear Algebra Appl.}, 57:105--113, 1984.

\bibitem{Bier1988}
Thomas Bier and Philippe Delsarte.
\newblock Some bounds for the distribution numbers of an association scheme.
\newblock {\em European J. Combin.}, 9(1):1--5, 1988.

\bibitem{Bier1987}
Thomas Bier and Nachimuthu Manickam.
\newblock The first distribution invariant of the {J}ohnson-scheme.
\newblock {\em Southeast Asian Bull. Math.}, 11(1):61--68, 1987.

\bibitem{Blinovsky2014}
Vladimir Blinovsky.
\newblock Minimal number of edges in hypergraph guaranteeing perfect fractional
  matching and {MMS} conjecture.
\newblock {\em ArXiv e-prints}, 2014.
\newblock arXiv:1310.0989v4 [math.CO].

\bibitem{Chowdhury2013}
Ameera Chowdhury, Ghassan Sarkis, and Shahriar Shahriari.
\newblock The {M}anickam-{M}iklós-{S}inghi conjectures for sets and vector
  spaces.
\newblock {\em ArXiv e-prints}, 2013.
\newblock arXiv:1309.2212v3 [math.CO].

\bibitem{Delsarte1976}
Philippe Delsarte.
\newblock Association schemes and {$t$}-designs in regular semilattices.
\newblock {\em J. Combinatorial Theory Ser. A}, 20(2):230--243, 1976.

\bibitem{Eisfeld1999}
J{\"o}rg Eisfeld.
\newblock The eigenspaces of the {B}ose-{M}esner algebras of the association
  schemes corresponding to projective spaces and polar spaces.
\newblock {\em Des. Codes Cryptogr.}, 17(1-3):129--150, 1999.

\bibitem{ErdHos1961}
P.~Erd{\H{o}}s, Chao Ko, and R.~Rado.
\newblock Intersection theorems for systems of finite sets.
\newblock {\em Quart. J. Math. Oxford Ser. (2)}, 12:313--320, 1961.

\bibitem{Frankl1986}
P.~Frankl and R.~M. Wilson.
\newblock The {E}rd{\H o}s-{K}o-{R}ado theorem for vector spaces.
\newblock {\em J. Combin. Theory Ser. A}, 43(2):228--236, 1986.

\bibitem{Hartke2014}
Stephen~G. Hartke and Derrick Stolee.
\newblock A linear programming approach to the {M}anickam-{M}ikl\'os-{S}inghi
  conjecture.
\newblock {\em European J. Combin.}, 36:53--70, 2014.

\bibitem{Hsieh1975}
W.~N. Hsieh.
\newblock Intersection theorems for systems of finite vector spaces.
\newblock {\em Discrete Math.}, 12:1--16, 1975.

\bibitem{Huang2014}
Hao Huang and Benny Sudakov.
\newblock The minimum number of nonnegative edges in hypergraphs.
\newblock {\em Electronic Journal of Combinatorics}, 2014.
\newblock . To appear.

\bibitem{Manickam1988}
N.~Manickam and D.~Mikl{\'o}s.
\newblock On the number of nonnegative partial sums of a nonnegative sum.
\newblock In {\em Combinatorics ({E}ger, 1987)}, volume~52 of {\em Colloq.
  Math. Soc. J\'anos Bolyai}, pages 385--392. North-Holland, Amsterdam, 1988.

\bibitem{Manickam1988a}
N.~Manickam and N.~M. Singhi.
\newblock First distribution invariants and {EKR} theorems.
\newblock {\em J. Combin. Theory Ser. A}, 48(1):91--103, 1988.

\bibitem{Pokrovskiy2013}
Alexey Pokrovskiy.
\newblock A linear bound on the {M}anickam-{M}iklos-{S}inghi conjecture.
\newblock {\em ArXiv e-prints}, 2013.
\newblock arXiv:1308.2176v1 [math.CO].

\bibitem{Vanhove2011}
Fr{\'e}d{\'e}ric Vanhove.
\newblock {\em Incidence geometry from an algebraic graph theory point of
  view}.
\newblock PhD thesis, University Of Ghent, 2011.

\bibitem{Wilson1984}
Richard~M. Wilson.
\newblock The exact bound in the {E}rd{\H o}s-{K}o-{R}ado theorem.
\newblock {\em Combinatorica}, 4(2-3):247--257, 1984.

\end{thebibliography}

\end{document}